\documentclass{amsart}
\usepackage[dvipdfmx]{graphicx}
\usepackage{url}
\usepackage{amssymb}
\usepackage{amsmath}
\usepackage{amssymb}
\usepackage{url}
\usepackage{cite}
\usepackage{etoolbox}
\let\bbordermatrix\bordermatrix
\patchcmd{\bbordermatrix}{8.75}{4.75}{}{}
\patchcmd{\bbordermatrix}{\left(}{\left[}{}{}
\patchcmd{\bbordermatrix}{\right)}{\right]}{}{}
\patchcmd{\bbordermatrix}{\begingroup}{\begingroup\openup1\jot}{}{}
\makeatletter
\patchcmd{\bbordermatrix}
  {\vcenter{\kern-\ht\@ne\unvbox\z@\kern-\baselineskip}}
  {\vcenter{\kern-\ht\@ne\unvbox\z@\kern-\baselineskip\kern2pt}}
  {}{}
\makeatother

\newtheorem{theorem}{Theorem}[section]
\theoremstyle{plain}

\newtheorem{corollary}{Corollary}[section]

\newtheorem{lemma}{Lemma}[section]

\newtheorem{problem}{Problem}
\newtheorem{proposition}{Proposition}

\newtheorem{fact}{Fact}

\numberwithin{equation}{section}
\begin{document}
\title[ Some results concerning the valences of (super) edge-magic graphs ]{ Some new results concerning the valences of (super) edge-magic graphs }
\author{Y. Takahashi}
\address{Department of Science and Engineering, Faculty of Electronics and
Informatics, Kokushikan University, 4-28-1 Setagaya, Setagaya-ku, Tokyo
154-8515, Japan}
\email{takayu@kokushikan.ac.jp}
\author{F.A. Muntaner-Batle}
\address{Graph Theory and Applications Research Group, School of Electrical
Engineering and Computer Science, Faculty of Engineering and Built
Environment, The University of Newcastle, NSW 2308 Australia }
\email{famb1es@yahoo.es}
\author{R. Ichishima}
\address{Department of Sport and Physical Education, Faculty of Physical
Education, Kokushikan University, 7-3-1 Nagayama, Tama-shi, Tokyo 206-8515,
Japan}
\email{ichishim@kokushikan.ac.jp}
\date{June 28, 2023}

\begin{abstract}
A graph $G$ is called edge-magic if there exists a bijective function $f:V\left(G\right) \cup E\left(G\right)\rightarrow \left\{1, 2, \ldots , \left\vert V\left( G\right) \right\vert +\left\vert E\left( G\right) \right\vert \right\}$ such that $f\left(u\right) + f\left(v\right) + f\left(uv\right)$ is a constant (called the valence of $f$) for each $uv\in E\left( G\right) $. 
If $f\left(V \left(G\right)\right) =\left\{1, 2, \ldots , \left\vert V\left( G\right) \right\vert \right\}$, then $G$ is called a super edge-magic graph. 
A stronger version of edge-magic and super edge-magic graphs appeared when the concepts of perfect edge-magic and perfect super edge-magic graphs were introduced. 
The super edge-magic deficiency $ \mu_{s}\left(G\right)$ of a graph $G$ is defined to be either the smallest nonnegative integer $n$ with the property that $G \cup nK_{1}$ is super edge-magic or $+ \infty$ if there exists no such integer $n$. 
On the other hand, the edge-magic deficiency $ \mu\left(G\right)$ of a graph $G$ is the smallest nonnegative integer $n$ for which $G\cup nK_{1}$ is edge-magic, being $ \mu\left(G\right)$ always finite. 
In this paper, the concepts of (super) edge-magic deficiency are generalized using the concepts of perfect (super) edge-magic graphs. 
This naturally leads to the study of the valences of edge-magic and super edge-magic labelings. 
We present some general results in this direction and study the perfect (super) edge-magic deficiency of the star $K_{1,n}$.
\end{abstract}

\maketitle

\section{Introduction}
Unless stated otherwise, the graph-theoretical notation and terminology used here will follow Chartrand and Lesniak \cite{CL}. 
In particular, the \emph{vertex set} of a graph $G$ is denoted by $V \left(G\right)$, while the \emph{edge set} of $G$ is denoted by $E\left (G\right)$.

For the sake of brevity, we will use the notation $\left[ a, b\right] $ for the interval of integers $x $ such that $a\leq x\leq b$.
Kotzig and Rosa \cite{KR} initiated the study of what they called magic valuations. 
This concept was later named edge-magic labelings by Ringel and Llad\'{o} \cite{RL}, and this has become the popular term.
A graph $G$ is called \emph{edge-magic} if there exists a bijective function $f:V\left(G\right) \cup E\left(G\right)\rightarrow \left[1, \left\vert V\left( G\right) \right\vert +\left\vert E\left( G\right) \right\vert \right]$ such that $f\left(u\right) + f\left(v\right) + f\left(uv\right)$ is a constant (called the \emph{valence} $\text{val}\left(f\right)$ of $f$) for each $uv\in E\left( G\right) $. Such a function is called an \emph{edge-magic labeling}. More recently, they have also been referred to as edge-magic total labelings by Wallis \cite{Wallis}.

Enomoto et al. \cite{ELNR} introduced a particular type of edge-magic labelings, namely, super edge-magic labelings. 
They defined an edge-magic labeling of a graph $G$ with the additional property that 
$f\left(V \left(G\right)\right)=\left[1, \left\vert V\left( G\right) \right\vert \right]$ to be a \emph{super edge-magic labeling}. 
Thus, a \emph{super edge-magic graph} is a graph that admits a super edge-magic labeling. 

Lately, super edge-magic labelings and super edge-magic graphs were called by Wallis \cite{Wallis} strong edge-magic total labelings and strongly edge-magic graphs, respectively. According to the latest version of the survey on graph labelings by Gallian \cite{Gallian} available to the authors, Hegde and Shetty \cite{HS2} showed that the concepts of super edge-magic graphs and strongly indexable graphs (see \cite{AH} for the definition of a strongly indexable graph) are equivalent.

The following result found in \cite{FIM} provides necessary and sufficient conditions for a graph to be super edge-magic, which will prove later to be useful.

\begin{lemma}
\label{trivial}
A graph $G$ is super edge-magic if and only if
there exists a bijective function $f:V\left( G\right) \rightarrow \left[ 1, \right. \left\vert V\left( G\right) \right\vert] $ such that the set 
\begin{equation*}
S=\left\{ f\left( u\right) +f\left( v\right): uv\in E\left( G\right) \right\}
\end{equation*}%
consists of $\left\vert E\left( G\right) \right\vert $ consecutive integers. In such a case, $f$ extends to a super
edge-magic labeling of $G$ with valence $k=\left\vert V\left( G\right) \right\vert +\left\vert E\left( G\right) \right\vert +s$, where $s=\min
\left( S\right) $ and 
\begin{equation*}
S=\left[ k-\left( \left\vert V\left( G\right) \right\vert +\left\vert E\left( G\right) \right\vert \right) ,k-\left( \left\vert V\left( G\right) \right\vert +1\right) \right] \text{.}
\end{equation*}
\end{lemma}

The \emph{low characteristic} and \emph{high characteristic} of a super edge-magic graph $G$ are defined by $\gamma\left(G\right) =\min\left(S\right)$ and $\Gamma\left(G\right) =\max\left(S\right)$, respectively, where $S$ is the set as in Lemma \ref{trivial}. These concepts will prove to be useful in this paper later.

For every graph $G$, Kotzig and Rosa \cite{KR} proved that there exists an edge-magic graph $H$ such that $H= G \cup nK_{1}$ for some nonnegative integer $n$. 
This motivated them to define the edge-magic deficiency of a graph. 
The \emph{edge-magic deficiency} $\mu \left(G\right)$ of a graph $G$ is the smallest nonnegative integer $n$ for which $G \cup nK_{1}$ is edge-magic. Inspired by Kotzig and Rosa's notion, the concept of \emph{super edge-magic deficiency} $\mu_{s} \left(G\right)$ of a graph $G$ was analogously defined in \cite{FIM2} as either the smallest nonnegative integer $n$ with the property that $G \cup nK_{1}$ is super edge-magic or $+ \infty$ if there exists no such integer $n$.  Thus, the super edge-magic deficiency of a graph $G$ is a measure of how “close” (“ far ”) $G$ is to (from) being super edge-magic. 

An alternative term exists for the super edge-magic deficiency, namely, the vertex dependent characteristic. 
This term was coined by Hedge and Shetty \cite{HS}. 
In \cite{HS}, they gave a construction of polygons having the same angles and distinct sides using the result on the super edge-magic deficiency of cycles provided in \cite{FIM3}.

Noting that for a super edge-magic labeling $f$ of a graph $G$ with order $p$ and size $q$, the valence $k$ is given by the formula:

\begin{equation*}
  k = \frac{\sum_{u \in V\left(G\right)} \text{deg}(u)f(u) + \sum_{i=p+1}^{p+q} i }{q} \text{.} \\
\end{equation*}
L\'{o}pez et al. \cite{LMR1} defined the set

\begin{multline*}
  S_{G} = \Biggl\lbrace
  \frac{\sum_{u \in V\left(G\right)} \text{deg}(u)g(u) + \sum_{i=p+1}^{p+q} i }{q} \text{:}  \\ 
   \text{ the function } g:V\left(G\right) \rightarrow \left[1,p\right] \text{ is bijective}
  \Biggr\rbrace \text{.}
\end{multline*}
If $\lceil \min S_{G} \rceil \leq \lfloor \max S_{G} \rfloor$, then the \emph{super edge-magic interval} of $G$ is the set
\begin{equation*}
  I_G = [\lceil \min S_{G} \rceil, \lfloor \max S_{G} \rfloor] \text{.}
\end{equation*}
The \emph{super edge-magic set} of $G$ is
\begin{equation*}
  \sigma_{G} = \left\lbrace k \in I_G \right. \text{: there exists a super edge-magic labeling of } G \text{ with valence } k \left. \right\rbrace.
\end{equation*}
L\'{o}pez et al. called a graph $G$ \emph{perfect super edge-magic} if $I_G = \sigma_G$. They showed that the family of paths $P_n$ is a family of perfect super edge-magic graphs with $|I_{P_{n}}|=1$ if $n$ is even and $|I_{P_{n}}| = 2$  if $n$ is odd, and raise the question of whether there is an infinite family of graphs $(F_1, F_2, \ldots)$  such that each member of the family is perfect super edge-magic and $\text{lim}_{i \rightarrow +\infty} |I_{F_{i}}| = +\infty$. They showed that graphs $G \cong C_{p^{k}} \odot \overline{K_{n}}$, where $p > 2$ is a prime number and $ \odot$ denotes the corona product, is such a family. For more detailed information on this matter, see \cite{LMR1}.

For an edge-magic labeling $f$ of a graph $G$, the valence $k$ is given by the formula:
\begin{equation*}
k = \frac{\sum_{u \in V\left(G\right)} \text{deg}(u)f(u) + \sum_{e\in E\left(G\right)} f\left(e\right) }{\left\vert E\left( G\right) \right\vert} \text{.}
\end{equation*}

In \cite{LMR4} L\'{o}pez et al. introduced, in a similar way of what we have seen so far for the case of super edge-magic labelings, the concepts of edge-magic interval, edge-magic set and perfect edge-magic graph.

For a graph $G$, define the set 
\begin{multline*}
  T_{G} = \Biggl\lbrace
  \frac{\sum_{u \in V\left(G\right)} \text{deg}(u)g(u) + \sum_{e \in E\left(G\right)}g\left(e\right) }{ \left\vert E\left( G\right) \right\vert} \text{:}\\
  \text{ the function } g:V\left(G\right)\cup E\left(G\right) \rightarrow \left[1, \left\vert V\left( G\right) \right\vert+ \left\vert E\left( G\right) \right\vert\right] \text{ is bijective}
  \Biggr\rbrace \text{.}
\end{multline*}
If $\lceil \min T_{G} \rceil \leq \lfloor \max T_{G} \rfloor$, then the \emph{edge-magic interval} of $G$ is the set
\begin{equation*}
   \lambda_{G} =\left[\lceil \min T_{G} \rceil, \lfloor \max T_{G} \rfloor \right] \text{.}
\end{equation*}
The \emph{edge-magic set} of $G$ is
\begin{equation*}
  \tau_{G}= \left\lbrace k \in \lambda_G \right. \text{: there exists an edge-magic labeling of } G \text{ with valence } k \left. \right\rbrace.
\end{equation*}
L\'{o}pez et al. called a graph $G$ \emph{perfect edge-magic} if $\lambda_{G} = \tau_{G}$.

Motivated by the concepts of edge-magic and super edge-magic deficiencies together with the concepts of perfect edge-magic and perfect super edge-magic graphs, we introduce the concepts of perfect edge-magic deficiency and perfect super edge-magic deficiency next. 

The \emph{perfect edge-magic deficiency} $\mu_{p}\left(G\right)$ of a graph $G$ is defined to be the smallest nonnegative integer $n$ with the property that $G\cup nK_{1}$ is perfect edge-magic or $ +\infty$ if there exists no such integer $n$. On the other hand, the \emph{perfect super edge-magic deficiency} $\mu_{p}^{s}\left(G\right)$ of a graph $G$ is defined to be the smallest nonnegative integer $n$ with the property that $G\cup nK_{1}$ is perfect super edge-magic or $ +\infty$ if there exists no such integer $n$.

In \cite{LMR4} L\'{o}pez et al. defined the \emph{irregular crown} $C(n; j_1, j_2, \ldots, j_n) = (V, E)$, where $n > 2$ and $j_i \geq 0$ for all $i \in \lbrack 1, n \rbrack$ as follows: $V = \lbrace v_{i}:i\in \lbrack 1,n \rbrack \rbrace \cup V_1 \cup V_2 \cup \cdots \cup V_n$, where $V_k = \lbrace v_{k}^{i}:i\in \lbrack 1,{j_{k} \rbrack} \rbrace$, if $j_k \neq 0$ and $V_k = \emptyset$ if $j_k = 0$ for each $k \in \lbrack 1, n \rbrack$ and $E = \lbrace v_{i}v_{i+1}: i \in \lbrack 1, n-1 \rbrack \rbrace \cup \lbrace v_{1}v_{n} \rbrace \cup (\cup^{n}_{k=1, j_{k} \neq 0} \lbrace v_{k}v_{k}^{l}: k \in \lbrack 1, j_{k} \rbrack \rbrace)$.
In particular, they denoted the graph $C_{m}^{n} \cong C(m; j_1, j_2, \ldots, j_m)$, where $j_{2i-1} = n$ for each $i \in \lbrack 1, (m+1)/2 \rbrack$, and $j_{2i} = 0$ for each $i \in \lbrack 1, (m-1)/2 \rbrack$ using the notation $C_{m}^{n}$. They proved that the graphs $C_{3}^{n}$ and $C_{5}^{n}$ are perfect edge-magic for all integers $n > 1$. 
In the same paper, they also proved that if $m=p^{k}$ when $p$ is a prime number and $k \in \mathbb{N}$, 
then the graph $C_{m} \odot \overline{K_{n}}$ is perfect edge-magic for all positive integers $n$.

L\'{o}pez et al. \cite{LMR3} defined the concepts of $\mathfrak{F}^{k}$-family and $\mathfrak{E}^{k}$-family of graphs as follows. The infinite family of graphs $(F_1, F_2, \ldots)$ is an $\mathfrak{F}^{k}$-\emph{family} if each element $F_n$ admits exactly $k$ different valences for super edge-magic labelings, and $\text{lim}_{n \rightarrow +\infty} |I(F_{n})| = +\infty$. The infinite family of graphs $(F_1, F_2, \ldots)$ is an $\mathfrak{E}^{k}$-\emph{family} if each element $F_n$ admits exactly $k$ different valences for edge-magic labelings, and $\text{lim}_{n \rightarrow +\infty} |J(F_{n})| = +\infty$.

An easy observation from the results found in \cite{On_paper} and independently in \cite{Wallis}, is that $(K_{1,2}, K_{1,3}, \ldots)$ is an $\mathfrak{F}^{2}$-family and $\mathfrak{E}^{3}$-family. They posed the following two problems: for which positive integers $k$ is it possible to find $\mathfrak{F}^{k}$-families and $\mathfrak{E}^{k}$-families? Their main results in \cite{LMR3} are that an $\mathfrak{F}^{k}$-family exits for each $k =$ 1, 2 and 3, and an $\mathfrak{E}^{k}$-family exits for each $k =$ 3, 4 and 7.

The following inequality was found by Enomoto et al. \cite{ELNR}.
\begin{theorem}
If $G$ is a super edge-magic graph, then 
\begin{equation*}
\left\vert E\left( G\right) \right\vert \leq 2\left\vert V\left( G\right) \right\vert -3 \text{.}
\end{equation*}
\end{theorem}

The following result found in \cite{IMR} provides a sufficient condition for a super edge-magic graph to contain a triangle.

\begin{theorem}
If $G$ is a super edge-magic graph with 
\begin{equation*}
\left\vert E\left( G\right) \right\vert = 2\left\vert V\left( G\right) \right\vert -3 \text{ or } 2\left\vert V\left( G\right) \right\vert -4 \text{,}
\end{equation*} then $G$ contains a triangle.
\end{theorem}

\section{General results}
In this section, we will establish some general results on super edge-magic graphs. 
We begin with the following result, which relates the order, girth and the cardinality of $\sigma_{G}$.

\begin{theorem}
\label{q=2p-3}
Let $G$ be a super edge-magic graph of size $2\left\vert V\left( G\right) \right\vert-3$ 
and girth $g\left(G\right)=3$. 
Then $\left\vert \sigma_{G} \right\vert=1$.
\end{theorem}

\begin{proof}
Consider such a graph $G$. 
If $f$ is a super edge-magic labeling of $G$, then 
\begin{equation*}
\left\{ f\left( u\right) +f\left( v\right): uv\in E\left( G\right) \right\}=\left[ 3, 2\left\vert V\left( G\right) \right\vert-1\right]
\end{equation*}
and 
\begin{equation*}
\text{val}\left(f\right)=\left\vert V\left( G\right) \right\vert +\left\vert E\left( G\right) \right\vert +\min\left\{ f\left( u\right) +f\left( v\right) : uv\in E\left( G\right) \right\}=3\left\vert V\left( G\right) \right\vert
\end{equation*}
by Lemma \ref{trivial}.
Therefore, $\sigma_{G} =\left\{ 3\left\vert V\left( G\right) \right\vert \right\}$, 
completing the proof.
\end{proof}

The next concept will prove to be useful throughout this paper. 
Let $f$ be a super edge-magic labeling of a graph $G$. Then the \emph{complementary labeling} $\overline{f}$ of $f$ is defined as 
\begin{equation*}
\overline{f}\left(x\right)=\left\{ 
\begin{tabular}{ll}
$\left(\left\vert V\left( G\right) \right\vert +1\right)-f\left(x\right)$ & if $x \in V\left(G\right)$ \\ 
$\left(2\left\vert V\left( G\right) \right\vert +\left\vert E\left( G\right) \right\vert +1\right)-f\left(x\right)$ & if $x \in E\left(G\right)$.%
\end{tabular}%
\right.
\end{equation*}

It is true that if $f$ is a super edge-magic labeling of a graph $G$, then $\overline{f}$ is also a super edge-magic labeling of $G$. 

Next, we show an example of a super edge-magic labeling of a graph $G$ with 
$\left\vert E\left( G\right) \right\vert =2\left\vert V\left( G\right) \right\vert -4$ and labeled with two different super edge-magic labelings $f$ and $\overline{f}$ (see Figure \ref{fig06} and Figure \ref{fig07}, respectively), producing two different valences.
On the orther hand, if $f$ is an edge-magic labeling of a graph $G$, then the complementary labeling $\overline{f}$ of $f$ is the labeling
$\overline{f}:V\left(G\right) \cup E\left(G\right)\rightarrow \left[1, \left\vert V\left( G\right) \right\vert +\left\vert E\left( G\right) \right\vert \right]$
defined by 
\begin{equation*}
\overline{f}\left(x\right)=\left(\left\vert V\left( G\right) \right\vert +\left\vert E\left( G\right) \right\vert +1\right)-f\left(x\right)
\end{equation*}
if $x \in V\left(G\right)\cup E\left(G\right)$.
Observe that the complementary labeling of an edge-magic labeling is also an edge-magic labeling.

\begin{figure}[ht]
  \begin{center}
    \includegraphics[keepaspectratio, width=0.3\linewidth]{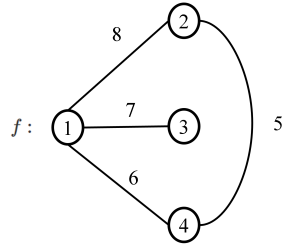}
    \caption{A super edge-magic labeling $f$ with valence 11}
    \label{fig06}
  \end{center}
\end{figure}

\begin{figure}[ht]
  \begin{center}
    \includegraphics[keepaspectratio, width=0.3\linewidth]{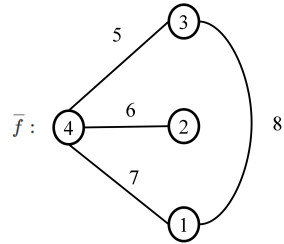}
    \caption{A super edge-magic labeling $\overline{f}$ with valence 12}
    \label{fig07}
  \end{center}
\end{figure}

\begin{theorem}
\label{q=2p-4}
Let $G$ be a super edge-magic graph of size
$2\left\vert V\left( G\right) \right\vert-4$. 
Then $\left\vert \sigma_{G} \right\vert=2$.
\end{theorem}

\begin{proof}
Let $G$ be a super edge-magic graph with 
$\left\vert E\left( G\right) \right\vert=2\left\vert V\left( G\right) \right\vert-4$ and $f$ be a super edge-magic labeling of $G$. 
Then it is clear that 
\begin{equation*}
\left\{ f\left( u\right) +f\left( v\right) : uv\in E\left( G\right) \right\}=\left[ 3, 2\left\vert V\left( G\right) \right\vert-2\right]
\end{equation*}
or
\begin{equation*}
\left\{ f\left( u\right) +f\left( v\right) : uv\in E\left( G\right) \right\}=\left[ 4, 2\left\vert V\left( G\right) \right\vert-1\right] \text{.}
\end{equation*}
Since each one of these two possibilities will provide a different valence, 
it follows that $\left\vert \sigma_{G} \right\vert \leq 2$.
However, we can guarantee that $\left\vert \sigma_{G} \right\vert=2$.
To see this, notice that the only way to get $3$ as an induced edge sum is by joining vertices, which have been labeled $1$ and $2$ by a super edge-magic labeling $f$ of $G$.
Now, if we consider the super edge-magic labeling $\overline{f}$, the vertices to which $f$ assignes labels $1$ and $2$ become labeled 
$\left\vert V\left( G\right) \right\vert $ and $\left\vert V\left( G\right) \right\vert  -1$ by $\overline{f}$, respectively.
Moreover, since they are adjacent, it follows that there is an edge with induced sum $2\left\vert V\left( G\right) \right\vert-1$ when we consider $\overline{f}$.
Thus, $f$ and $\overline{f}$ have valences 
\begin{equation*}
\text{val}\left(f\right)=3+\left\vert V\left( G\right) \right\vert +2 \left\vert V\left( G\right) \right\vert -4=3\left\vert V\left( G\right) \right\vert -1
\end{equation*}
and 
\begin{equation*}
\text{val}\left(\overline{f}\right)=\left(2\left\vert V\left( G\right) \right\vert -1\right)+\left\vert V\left( G \right) \right\vert +1=3\left\vert V\left( G\right) \right\vert \text{,}
\end{equation*}
respectively.
In a similar way, if $f$ is a labeling that has an induced edge sum 
$2\left\vert V\left( G\right) \right\vert -1$, then vertices labeled $\left\vert V\left( G\right) \right\vert $ and 
$\left\vert V\left( G\right) \right\vert -1$ must be adjacent.
Therefore, taking the complement of this labeling, we obtain the desired result.
\end{proof}

The following result was shown in \cite{IMR} (see also \cite{LM} for a different approach to this problem using the product $\otimes_{h}$ defined on digraphs in \cite{FIMR}). 

\begin{theorem}
There exists an infinite family of super edge-magic graphs $G$ of size $2\left\vert V\left( G\right) \right\vert -5$ and girth $g\left(G\right)=5$.
\end{theorem}

We next show the following result.

\begin{theorem}
Let $G$ be a super edge-magic graph with $\left\vert E\left( G\right) \right\vert =2\left\vert V\left( G\right) \right\vert -5$ and girth $g\left(G\right) \geq 5$. 
Then $\left\vert \sigma_{G} \right\vert=1$.
\end{theorem}

\begin{proof}
Let $G$ be a graph of girth $g\left(G\right) \geq 5$ and with a super edge-magic labeling $f$. 
Consider the set
$S=\left\{ f\left( u\right) +f\left( v\right) : uv\in E\left( G\right) \right\}$. 
Then there are three possibilities for $S$, namely, 
\begin{equation*}
S=\left[ 3, 2\left\vert V\left( G\right) \right\vert-3\right],
S=\left[ 4, 2\left\vert V\left( G\right) \right\vert-2\right] \text{ or } 
S=\left[ 5, 2\left\vert V\left( G\right) \right\vert-1\right] \text{.}
\end{equation*}

Assume, to the contrary, that $S=\left[ 3, 2\left\vert V\left( G\right) \right\vert-3\right]$. 
Now, the only way to get the induced edge sum $3$ is by joining vertices labeled $1$ and $2$. 
The only way to get the induced edge sum $4$ is by joining vertices labeled $1$ and $3$. 
Hence, $G$ has edges joining those vertices. 
To get the induced edge sum $5$, we cannot use vertices labeled $2$ and $3$, since this choice would force a triangle and would have girth $3$. 
Thus, we need to include an edge joining vertices labeled by $1$ and $4$.
For a similar reason, to get the induced edge sum $6$, we need to include an edge joining vertices labeled by $1$ and $5$.
Any other choice would produce a triangle. 
We proceed in this manner until we get the induced edge sum 
$\left\vert V\left( G\right) \right\vert +1$.
At this point, we have a star of order $\left\vert V\left( G\right) \right\vert$ (see Figure \ref{fig01}).
Thus, we cannot add any new edge to get the induced edge sum $\left\vert V\left( G\right) \right\vert +2$, 
since any new edge would produce a triangle. 
In a similar way, we can see that the set $S=\left[ 5, 2\left\vert V\left( G\right) \right\vert-1\right]$ cannot take place.
Therefore, the only possibility left is $S=\left[ 4, 2\left\vert V\left( G\right) \right\vert-2\right]$, and the result follows.
\end{proof}

\begin{figure}[ht]
  \begin{center}
    \includegraphics[keepaspectratio, width=0.18\linewidth]{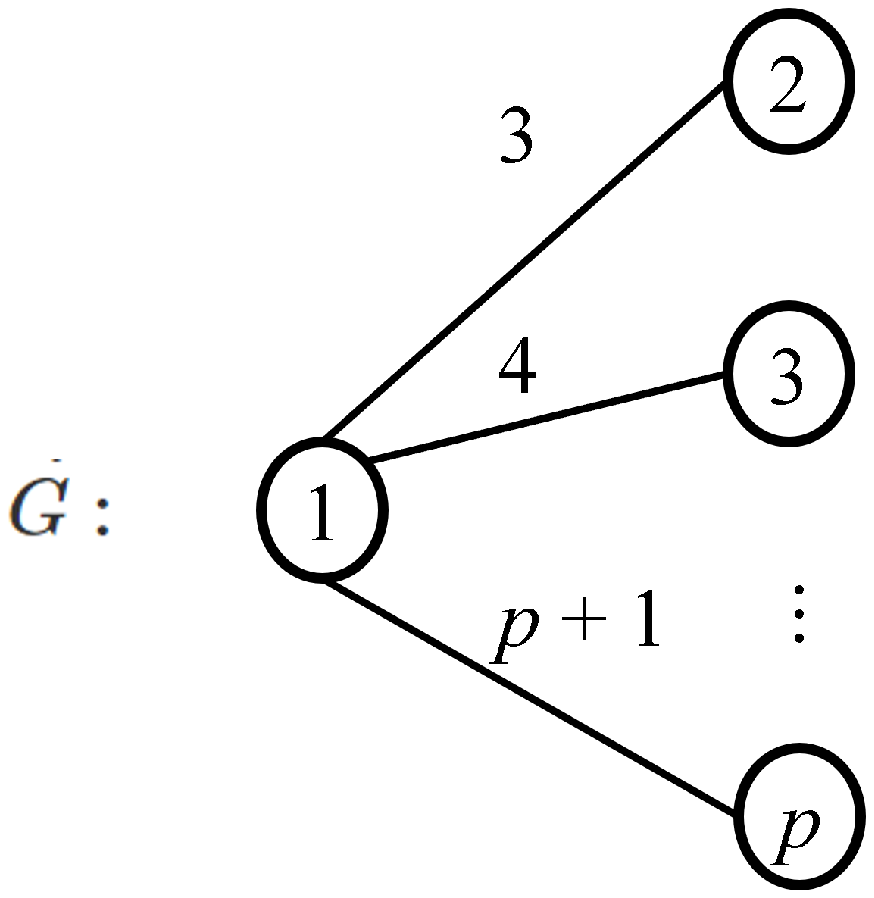}
    \caption{A star $G$ of order $\left\vert V\left( G\right) \right\vert=p$ with induced edge sums}
    \label{fig01}
  \end{center}
\end{figure}

To conclude this section, we present the following result.

\begin{theorem}
Let $G$ be a graph with $\left\vert E\left( G\right) \right\vert  \geq \left\vert V\left( G\right) \right\vert $. 
If $G$ is a super edge-magic graph of girth $g\left(G\right) \geq 4$, 
then $3+\left\vert E\left( G\right) \right\vert + \left\vert V\left( G\right) \right\vert $ is not a valence for any super edge-magic labeling of $G$.
\begin{proof}
Let $G$ be a super edge-magic graph with $\left\vert E\left( G\right) \right\vert \geq \left\vert V\left( G\right) \right\vert $, girth $g\left(G\right) \geq 4$
and with a super edge-magic labeling $f$.
Assume, to the contrary, that $\text{val}\left(f\right)= 3+\left\vert V\left( G\right) \right\vert + \left\vert E\left( G\right) \right\vert $.
Then 
\begin{equation*}
3+\left\vert V\left( G\right) \right\vert + \left\vert E\left( G\right) \right\vert =\min\left\{ f\left( u\right) +f\left( v\right) : uv\in E\left( G\right)
 \right\}+\left\vert V\left( G\right) \right\vert + \left\vert E\left( G\right) \right\vert  \text{,}
\end{equation*}
implying that 
\begin{equation*}
\min \left\{ f\left( u\right) +f\left( v\right) : uv\in E\left( G\right) \right\}=3   \text{.}
\end{equation*} 
Now, the only way to obtain an edge with induced sum $3$ is by joining vertices labeled $1$ and $2$. 
Also, the only way to obtain an edge with induced sum $4$ is by joining vertices labeled $1$ and $3$. 
Hence, if we want to avoid triangles, 
we need to induce the edge sums with labels 
$\{1,4\}, \{1,5\}, \dots, \{1, \left\vert V\left( G\right) \right\vert \}$ 
to obtain the induced edge sums $5,6,\dots,\left\vert V\left( G\right) \right\vert+1$, 
since any other choice would produce a triangle.
Moreover, notice that any choice for the edge with induced sum $\left\vert V\left( G\right) \right\vert+2$ will produce a triangle, which contradicts the fact that $g\left(G\right) \geq 4$.
\end{proof}
\end{theorem}

\section{New results involving stars}

This section is devoted to study the valences for the edge-magic and super edge-magic labelings of the unions of stars and isolated vertices. 
We begin by presenting a result concerning with isomorphic labelings. 
To do this, we introduce the concept of isomorphic (super) edge-magic labelings. 
For a graph $G$, assume that $f$ and $g$ are two (super) edge-magic labelings of $G$. 
We denote by $G_{f}$ and $G_{g}$ the graph $G$, where the vertices of $G$ take the name of the labels assigned by the labelings $f$ and $g$, respectively. 
Then we compute the adjacency matrices of $G_{f}$ and $G_{g}$, respectively, and we denote them by $A\left(G_{f}\right)$ and $A\left(G_{g}\right)$, where the rows and columns are placed in increasing order from left to right and top to bottom, respectively. Notice that the rows and columns are not necessarily labeled with consecutive integers, since the vertex labels of an edge-magic labeling are not necessarily consecutive integers. Then labelings $f$ and $g$ are \emph{isomorphic labelings}, written $f\cong g$, if $A\left(G_{f}\right)=A\left(G_{g}\right)$.

For example, consider the graph $G=C_{4}\cup K_{1}$ with two edge-magic labelings $f$ and $g$ illustrated in Figure \ref{fig04}.

\begin{figure}[ht]
  \begin{center}
    \includegraphics[keepaspectratio, width=0.7\linewidth]{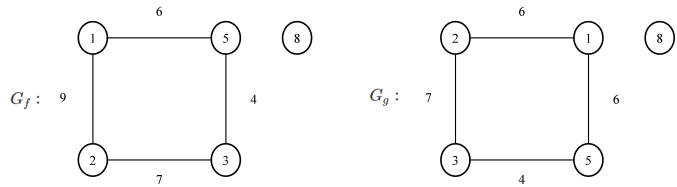}
    \caption{Two edge-magic labelings $f$ and $g$ of $G$}
    \label{fig04}
  \end{center}
\end{figure}

\noindent
Then 
\begin{equation*}
A\left(G_{f}\right)=A\left(G_{g}\right)=\bbordermatrix{ & 1 & 2 & 3 & 5 & 8 \cr 1 & 0 & 1 & 0 & 1 & 0 \cr 2 &1 & 0 & 1 & 0 & 0\cr3& 0 & 1 & 0 & 1 & 0 \cr 5& 1 & 0 & 1 & 0 & 0 \cr 8& 0 & 0 & 0 & 0 & 0 }
\end{equation*}
and hence $f\cong g$. 

Next, consider the graph $G=C_{4}\cup K_{1}$ labeled by $\overline{f}$ as illustrated in Figure \ref{fig05}.

\begin{figure}[ht]
  \begin{center}
    \includegraphics[keepaspectratio, width=0.4\linewidth]{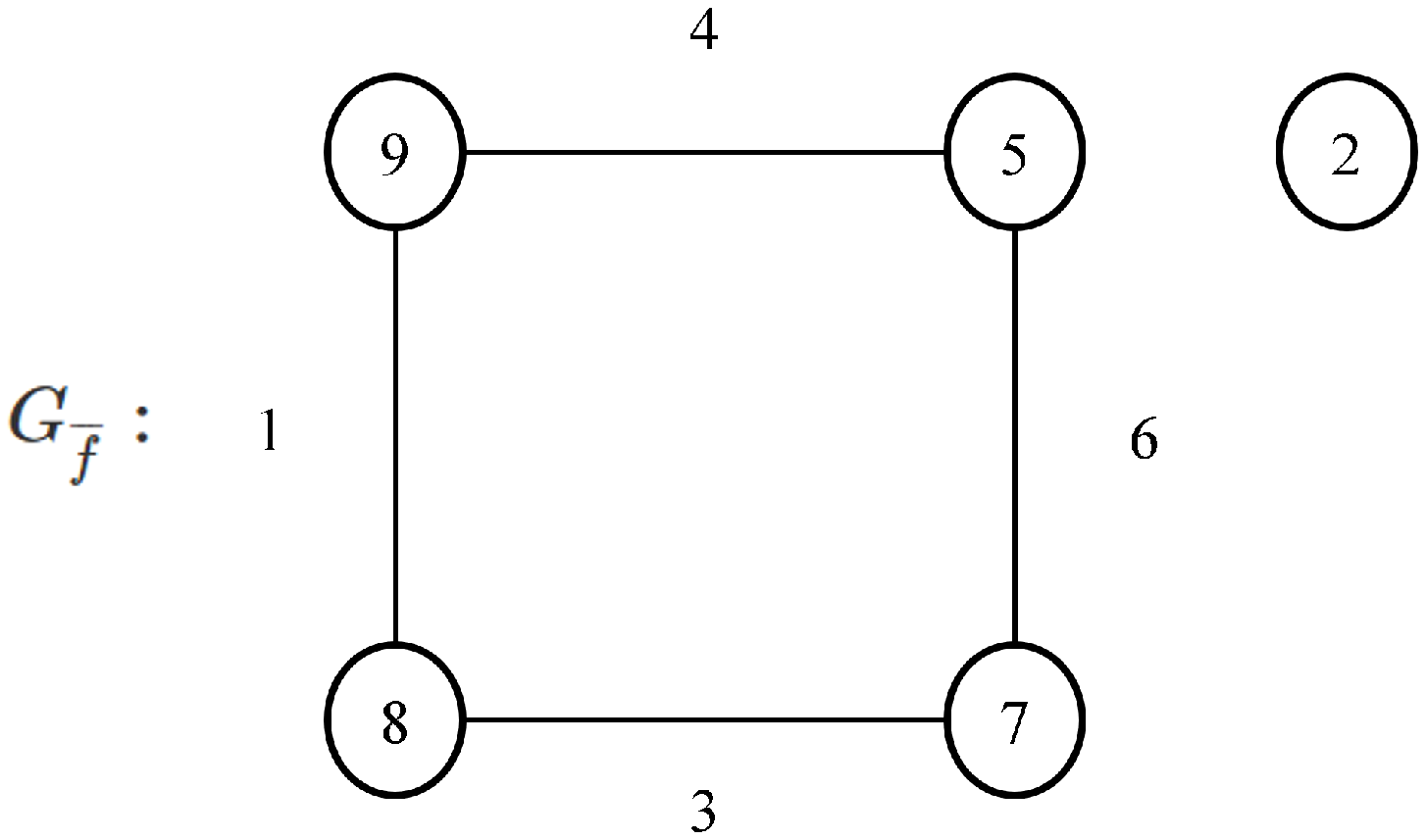}
    \caption{An edge-magic labeling $\overline{f}$ of $G$}
    \label{fig05}
  \end{center}
\end{figure}
\noindent
Then it is clear that $\overline{f}\not \cong f$ and $\overline{f}\not \cong g$, 
since 
\begin{equation*}
  A\left(G_{\overline{f}}\right)= \bbordermatrix{ & 2 & 5 & 7 & 8 & 9 \cr 2 & 0 & 0 & 0 & 0 & 0 \cr 5 & 0 & 0 & 1 & 0 & 1 \cr 7 & 0 & 1 & 0 & 1 & 0 \cr 8 & 0 & 0 & 1 & 0 & 1 \cr 9 & 0 & 1 & 0 & 1 & 0 } \text{,}
\end{equation*}
$A\left(G_{\overline{f}}\right) \not = A\left(G_{f}\right)$ and $A\left(G_{\overline{f}}\right) \not = A\left(G_{g}\right)$.

\begin{theorem}
For every two positive integers $n$ and $l$, there exist exactly $\left(l+1\right)\left(l+2\right)$ non-isomorphic super edge-magic labelings of $K_{1,n}\cup lK_{1}$.
\end{theorem}

\begin{proof}
Let $G=K_{1,n}\cup lK_{1}$, and define the graph $G$ with 
\begin{equation*}
V\left(G\right)=\{x\} \cup \{y_{i}: i\in \left[ 1, n\right] \}\cup \{z_{i}: i\in \left[ 1, l\right] \}
\end{equation*}
and $E\left(G\right)=\{xy_{i}: i\in \left[ 1, n\right] \}$.
By Lemma \ref{trivial}, the set $\{f\left(xy_{i}\right):  i\in \left[1,n\right]\}$ is a set of $n$ consecutive integers.
Hence, if $S=\{f\left(y_{i}\right):  i\in \left[1,n\right]\}$ is a set of $n$ consecutive integers, then
$S$ is $\left[ 1, n\right] \text{or} \left[ 2, n+1\right]$ \text{or} $\dots \text{ or } \left[ l+2, n+l+1\right]$.
This gives us $\left(l+2\right)$ possibilities for the labels of $\{y_{i}: i\in \left[ 1, n\right] \}$
up to reordering.
For each one of these possibilities, there are exactly $\left(l+1\right)$ possible labels that have not been used.
This means that these labels must be assigned to the vertices $\{x\} \cup \{z_{i}:i\in \left[1,l\right]\}$.
However, since $\deg z_{i}=0$ for each $i\in \left[1,l\right]$, we only need to concern about the label assigned to $x$.
Therefore, there are exactly $\left(l+1\right)$ possible choices for $f\left(x\right)$.
This produces exactly $\left(l+1\right)\left(l+2\right)$ non-isomorphic super edge-magic labelings of $G$.
\end{proof}

\begin{fact} By the proof of the previous result, we know that if $f$ is a super edge-magic labeling of $G$, 
then the set $\{f\left(y_{i}\right):  i\in \left[1,n\right]\}$ is a set of $n$ consecutive integers. 
\end{fact}

Let $f$ be any super edge-magic labeling of $G$, and assume, without loss of generality, that 
\begin{equation*}
f\left(y_{1}\right)< f\left(y_{2}\right)<\cdots<f\left(y_{n}\right)
\end{equation*}
and 
\begin{equation*}
f\left(z_{1}\right)< f\left(z_{2}\right)<\cdots<f\left(z_{n}\right) \text{.}
\end{equation*}
For the complementary labeling of $f$, we have 
\begin{equation*}
\overline{f}\left(y_{1}\right)> \overline{f}\left(y_{2}\right)>\cdots>\overline{f}\left(y_{n}\right)
\end{equation*}
and 
\begin{equation*}
\overline{f}\left(z_{1}\right)>\overline{f}\left(z_{2}\right)>\cdots>\overline{f}\left(z_{n}\right) \text{.}
\end{equation*}

Let $f$ be a super edge-magic labeling of $K_{1,n}\cup lK_{1}$. By Lemma \ref {trivial}, we know that
\begin{eqnarray*}
\text{val}\left(f\right)&=&f\left(x\right)+f\left(y_{1}\right)+2n+l+1\\
&=&f\left(x\right)+f\left(y_{n}\right)+n+2 \text{.}
\end{eqnarray*}
Thus, both sums $f\left(x\right)+f\left(y_{1}\right)$ and $f\left(x\right)+f\left(y_{n}\right)$
perfectly determine the valence of the labeling $f$.

From this, the following fact is clear.
\begin{fact}
\label{fact2}
Let $f_{1}$ and $f_{2}$ be two super edge-magic labelings of $G$. 
Then 
\begin{equation*}
\mathrm{val}\left(f_{1}\right)+1=\mathrm{val}\left(f_{2}\right)
\end{equation*}
if and only if 
\begin{equation*}
f_{1}\left(x\right)+f_{1}\left(y_{1}\right)+1= f_{2}\left(x\right)+f_{2}\left(y_{1}\right)
\end{equation*}
if and only if
\begin{equation*}
f_{1}\left(x\right)+f_{1}\left(y_{n}\right)+1= f_{2}\left(x\right)+f_{2}\left(y_{n}\right) \text{.}
\end{equation*}
\end{fact}
Let $f$ be any super edge-magic labeling of $G$. 
Then $\gamma\left(G\right)=f\left(x\right)+f\left(y_{1}\right)$ and $\Gamma\left(G\right)=f\left(x\right)+f\left(y_{n}\right)$.

The fact that the set $\{f\left(y_{i}\right):  i\in \left[1,n\right]\}$ is a set of $n$ consecutive integers for any super edge-magic labeling of $G$ suggests the following definitions. 
Consider a super edge-magic labeling $f$ of $G$ and define the set $S_{2}^{f}=\{f\left(y_{i}\right):  i\in \left[1,n\right]\}$.
Then 
\begin{equation*}
S_{1}^{f}=\left[1,f\left(y_{1}\right) \right] \text{ and }
S_{3}^{f}=\left[1,n+l+1\right] \backslash\left(S_{1}^{f} \cup S_{1}^{f}\right) \text{.}
\end{equation*}
A super edge-magic labeling $f$ of $G$ is of type 1 and we write it as $f\in T_{1}$ if $f\left(x\right)\in S_{1}^{f}$, and it is of type 2 and we write it as $f\in T_{2}$ if $f\left(x\right)\in S_{3}^{f}$.
From Fact 1, it is easy to deduce the following fact.

\begin{fact}
The set $T_{1}\cup T_{2}$ is a partition of the set 
\begin{equation*}
F \left(G\right)= \left\lbrace \right. f\mathrm{\ :\ } f \text{ is a super edge-magic labeling of } G\left. \right\rbrace \text{.}
\end{equation*}
\end{fact}

We are now ready to state and prove the following theorem.

\begin{theorem}
There exists a bijective function between $T_{1}$ and $T_{2}$.
\end{theorem}

\begin{proof}
Consider the function $\phi:T_{1} \rightarrow T_{2}$ defined by $\phi\left(f\right)= \overline{f}$ for all $f\in T_{1}$.
First, we show that if $f\in T_{1}$, then $\overline{f}\in T_{2}$.
If $f\in T_{1}$, then
\begin{equation*}
f\left(x\right)<f\left(y_{1}\right)< f\left(y_{2}\right)<\cdots<f\left(y_{n}\right) \text{,}
\end{equation*}
implying that 
\begin{equation*}
  \begin{split}
    \left(n+l+2\right)-f\left(x\right) & > 
    \left(n+l+2\right)-f\left(y_{1}\right)  \\
& > \left(n+l+2\right)-f\left(y_{2}\right)> \cdots>\left(n+l+2\right)-f\left(y_{n}\right) \text{.}
  \end{split}
\end{equation*}
Thus, 
\begin{equation*}
\overline{f}\left(x\right)>\overline{f}\left(y_{1}\right)> \overline{f}\left(y_{2}\right)>\cdots>\overline{f}\left(y_{n}\right)
\end{equation*}
so that $\overline{f}\in T_{2}$.

Next, assume that $\left\vert \{f_{a},f_{b}\}\cap T_{1}\right\vert=2$, and we will show that $\phi\left(f_{a}\right) \neq\phi\left(f_{b}\right)$. 
Assume that $f_{a} \neq f_{b}$. 
Then we have two possibilities. 

\noindent \textbf{Case 1.} If $f_{a} \neq f_{b}$, 
then  $\left(n+l+2\right)-\phi\left(f_{a}\right)\neq \left(n+l+2\right)-\phi\left(f_{b}\right)$.
Thus, $ \overline{f_{a}}\neq \overline{f_{b}}$ so that $\phi\left(f_{a}\right) \neq \phi\left(f_{b}\right)$.

\noindent \textbf{Case 2.} If $\{f_{a}\left(y_{i}\right):  i\in \left[1,n\right]\} \neq \{f_{b}\left(y_{i}\right):  i\in \left[1,n\right]\}$, then, without loss of generality, assume that
\begin{equation*}
f_{a}\left(y_{1}\right)< f_{b}\left(y_{1}\right) \text{ and } f_{a}\left(y_{n}\right)< f_{b}\left(y_{n}\right) \text{.}
\end{equation*}
Then 
\begin{equation*}
\left(n+l+2\right)-f_{a}\left(y_{1}\right)> 
\left(n+l+2\right)-f_{b}\left(y_{1}\right) 
\end{equation*}
and 
\begin{equation*}
\left(n+l+2\right)-f_{a}\left(y_{n}\right)>\left(n+l+2\right)- f_{b}\left(y_{n}\right) \text{.}
\end{equation*}
This implies that
$\overline{f_{a}}\left(y_{1}\right) >\overline{f_{b}}\left(y_{1}\right)$, 
which clearly implies that
$\overline{f_{a}}\left(y_{n}\right)> \overline{f_{b}}\left(y_{n}\right)$.
Since  
\begin{equation*}
\max\{\overline{f_{a}}\left(y_{i}\right):  i\in 
\left[1,n\right]\}>\max\{\overline{f_{b}}\left(y_{i}\right):  i\in \left[1,n\right]\}
\end{equation*}
and 
\begin{equation*}
\min\{\overline{f_{a}}\left(y_{i}\right):  i\in 
\left[1,n\right]\}>\min\{\overline{f_{b}}\left(y_{i}\right):  i\in \left[1,n\right]\} \text{,}
\end{equation*}
it follows that $\phi\left(f_{a}\right) \neq \phi\left(f_{b}\right)$.
Therefore, $\phi$ is injective. 
It now remains to see that $\phi$ is surjective. Let $\overline{f}\in T_{2}$ and, without loss of generality, assume that $\overline{f}$ has the property that 
\begin{equation*}
\overline{f}\left(x \right)>\overline{f}\left(y_{1}\right)>\cdots>\overline{f}\left(y_{n}\right) \text{,}
\end{equation*}
that is,
\begin{equation*}
\left(n+l+2\right)-\overline{f}\left(x\right)< \left(n+l+2\right)-\overline{f}\left(y_{1}\right)< \cdots<\left(n+l+2\right)-\overline{f}\left(y_{n}\right) \text{.}
\end{equation*}
This implies that 
\begin{equation*}
\overline{\overline{f}}\left(x\right)< \overline{\overline{f}}\left(y_{1}\right)< \cdots < \overline{\overline{f}}\left(y_{n}\right) \text{.}
\end{equation*}
Hence, $\overline{\overline{f}}\in T_{1}$. 
Clearly, $\phi\left(\overline{\overline{f}}\right)=\overline{f}$, and $\phi$ is surjective. 
Therefore, $\phi$ is bijective.
\end{proof}

\begin{proposition}
Let $f_{1}$ and $f_{2}$ be two super edge-magic labelings of $K_{1,n}\cup lK_{1}$ such that the valences of $f_{1}$ and $f_{2}$ are consecutive. 
Then the valences of  $\overline{f_{1}}$ and $\overline{f_{2}}$ are consecutive.
\end{proposition}

\begin{proof}
Assume that $\text{val}\left(f_{1}\right)$ and $\text{val}\left(f_{2}\right)$ are consecutive, and let $\text{val}\left(f_{1}\right)+1= \text{val}\left(f_{2}\right)$. 
Then we have
$f_{1}\left(a\right)+f_{1}\left(b\right)+f_{1}\left(ab\right)+1=f_{2}\left(a\right)+f_{2}\left(b\right)+f_{2}\left(ab\right)$,
where $ab\in E\left(K_{1,n}\cup lK_{1}\right)$.
If we let $\omega=n+l+2$ and $\rho=3n+2l+3$, then we have 
\begin{equation*}
-2\omega-\rho+f_{1}\left(a\right)+f_{1}\left(b\right)+f_{1}\left(ab\right)+1
=-2\omega-\rho +f_{2}\left(a\right)+f_{2}\left(b\right)+f_{2}\left(ab\right) \text{,}
\end{equation*}
implying that 
\begin{eqnarray*}
&&\left(\omega-f_{1}\left(a\right)\right)+\left(\omega-f_{1}\left(b\right)\right)+\left(\rho-f_{1}\left(ab\right)\right)-1 \\
&=&\left(\omega-f_{2}\left(a\right)\right)+\left(\omega-f_{2}\left(b\right)\right)+\left(\rho-f_{2}\left(ab\right)\right) \text{.}
\end{eqnarray*}
Hence, 
$\overline{f_{1}}\left(a\right)+ \overline{f_{1}}\left(b\right)+ \overline{f_{1}}\left(ab\right)-1
=\overline{f_{2}}\left(a\right)+ \overline{f_{2}}\left(b\right)+ \overline{f_{2}}\left(ab\right)$.
Thus, $\text{val}\left(\overline{f_{1}}\right)-1= \text{val}\left(\overline{f_{2}}\right)$.
Consequently, the valences of $\overline{f_{1}}$ and $\overline{f_{2}}$ are consecutive.
\end{proof}

Next, consider the graph $K_{1,n}\cup lK_{1}$ together with its associated set $T_{1}$.
Define the set $S\left(T_{1}\right)$ as
$S\left(T_{1}\right)= \left\lbrace \right.x\in \mathbb{N}: x= \text{val}\left(f\right) \text{ and } f\in T_{1} \left. \right\rbrace$. Also, consider the set $T_{2}$ associated to  $K_{1,n}\cup lK_{1}$ and, similarly, define the set  $S\left(T_{2}\right)$ as
$S\left(T_{2}\right)= \left\lbrace \right.x\in \mathbb{N}: x= \text{val}\left(f\right) \text{ and } f\in T_{2} \left. \right\rbrace$. 
With these definitions in hand, we have the following result.

\begin{theorem}
The set $S\left(T_{1}\right)$ consists of consecutive integers.
\end{theorem}

\begin{proof}
Let $f\in T_{1}$, and consider $\left(l+1\right)$ cases, depending on the possibilities for $S_{1}^{f}$. 

\noindent \textbf{Case 1.} Let $S_{1}^{f}=\{1\}$. 
In this case, $\min \left(S_{2}^{f}\right)=2$ and hence $3$ is the low characteristic of $f$.

\noindent \textbf{Case 2.} Let $S_{1}^{f}=\{1,2\}$. 
In this case, $\min \left(S_{2}^{f}\right)=3$ and hence the low characteristic of $f$ is $4$ or $5$.

\noindent \textbf{Case 3.} Let $S_{1}^{f}=\{1,2,3\}$. 
In this case, $\min \left(S_{2}^{f}\right)=4$ and hence the low characteristic of $f$ is $5$, $6$ or $7$.

\noindent \textbf{Case 4.} Let $S_{1}^{f}=\{1,2,3,4\}$. 
In this case, $\min \left(S_{2}^{f}\right)=5$ and hence the low characteristic of $f$ is $6$, $7$, $8$ or $9$.
\begin{equation*}
\vdots
\end{equation*}
\noindent \textbf{Case $\left(l+1\right)$.} Let $S_{1}^{f}=\{1,2,\dots, l+1\}$. In this case, $\min \left(S_{2}^{f}\right)=l+2$ and hence the low characteristic of $f$ is $l+3$, $l+4$, \dots, or $2l+3$.

\noindent This means that the labelings of $T_{1}$ have low characteristics $3$, $4$,\dots, $2l+3$. 
Therefore, we conclude by Fact \ref{fact2} that the set $S\left(T_{1}\right)$ is a set of consecutive integers.
\end{proof}

The following result is an immediate consequence of the preceding theorem and proposition.

\begin{corollary}
The set $S\left(T_{2}\right)$ consists of consecutive integers.
\end{corollary}

Next, we will concentrate on the elements of $S\left(T_{2}\right)$. 
We know that $S\left(T_{2}\right)$ contains the valences of the complementary labelings of the super edge-magic labelings of $T_{1}$. 
We have seen above that the low characteristics are $3,4,\dots,2l+3$. 
Thus, the high characteristics of the elements of $T_{2}$ are $\left(2n+2l+4\right)-3, \left(2n+2l+4\right)-4,\dots, \left(2n+2l+4\right)-\left(2l+3\right)$. 
From these, we deduce that the minimum among all the low characteristics of these labelings is 
\begin{equation*}
\left(2n+2l+4\right)-\left(2l+3\right)=n+2 \text{.}
\end{equation*}
Therefore, the only requirement needed for $K_{1,n}\cup lK_{1}$ to be perfect super edge-magic is $2l+3\geq n+2$.

\begin {theorem}
For every positive integer $n$, 
\begin{equation*}
\mu_{p}\left(K_{1,n}\right)<+ \infty \text{.}
\end{equation*}
\end {theorem}

\begin{proof}
Let $G=  K_{1,n}\cup lK_{1}$, and assume that $G$ is perfect super edge-magic. 
Then it is trivial to observe that $K_{1,n}\cup l ^{\prime}K_{1}$ is perfect super edge-magic for any $ l ^{\prime} $ with $ l ^{\prime} \geq  l$. 
Since $G$ is perfect super edge-magic, it follows that the minimum valence of a super edge-magic labeling (and also the smallest possible valence for an edge-magic labeling of $G$) is computed as follows.
First, observe that 
$\left\vert V\left( G\right) \right\vert =n+l+1$ and $\left\vert E\left( G\right) \right\vert =n$. 
Next, observe that the super edge-magic labeling with smallest valence has the following properties.
\begin{enumerate} 
\item The vertex of $G$ of degree $n$ is labeled $1$.
\item The labels of the vertices other than isolated vertices and edges are all the numbers in the set
$\left[2,2n\left(2n+1\right)\right]$.
\end{enumerate}
With this knowledge in hand, we can compute the minimum possible valence $\text{val}_{\min}\left(G\right)$ of $G$ as follows:
\begin{equation*}
\text{val}_{\min}\left(G\right) = \frac{n + \sum_{i=2}^{2n+1} i }{n}=2n+4 \text{.}
\end{equation*}
Thus, the minimum possible valence of any edge-magic labeling of $G$ is $2n+4$.
On the other hand, the maximum valence $\text{val}_{\max}\left(G\right)$ of $G$ is given by 
\begin{equation*}
\text{val}_{\max}\left(G\right) = \frac{\left(n-1\right)\left(n+l+1\right) + \sum_{i=l+1}^{2n+l+1} i }{n}=3n+3l+3 \text{.}
\end{equation*}
Hence, the maximum possible labeling of $G$ is $3n+3l+3$.
Since $G$ is perfect super edge-magic by assumption, it follows that for every $\alpha \in \left[2n+4,3n+3l+3\right]$,
there exists a super edge-magic labeling of $G$ that has valence $\alpha$.
At this point, define the graph $G$ with
\begin{equation*}
V\left(G\right)=\{u_{i}: i \in \left[1,l\right]\} \cup \{v_{i}: i \in \left[n+l+1,2n+l+1\right]
\end{equation*}
and $E\left(G\right)=\{v_{2n+l+1}v_{i}: i \in \left[n+l+1,2n+l\right]\}$,
and consider the following $l$ edge-magic labelings $f_{0},f_{1},\dots,f_{l-1}$.
For each $k \in \left[0,l-1\right]$, let $f_{k}\left(v_{i}\right)=i$, where $i \in \left[1,2n+l+1\right]$. 
Also, let $f_{0}\left(v_{2n+l+1}v_{n+l+i}\right)=n+l+1-i$ for each $i \in \left[1,n\right]$,
and the labels not used go on the set of vertices $\{u_{i}: i\in \left[1,l\right]\}$.
Now, the labels on the edges assigned by each edge-magic labeling $f_{k}$ ($k\in  \left[1,l\right]$) are defined as 
$f_{k}\left(ab\right)=f_{0}\left(ab\right)-k$ for $ab\in E\left(G\right)$.
Once again, the labels assigned by $f_{k}$ to the vertices in the set $\{u_{i}: i\in \left[1,l\right]\}$ are the labels not used for the rest of vertices and edges.
These edge-magic labelings $f_{0},f_{1},\dots,f_{l-1}$ produce the valences $4n+3l+2,4n+3l+1,\dots,4n+2l+2$ (see the following example for labelings $f_{0},f_{1},f_{2}$ depicted in Figure \ref{fig02}).

\begin{figure}[ht]
  \begin{center}
    \includegraphics[keepaspectratio, width=0.8\linewidth]{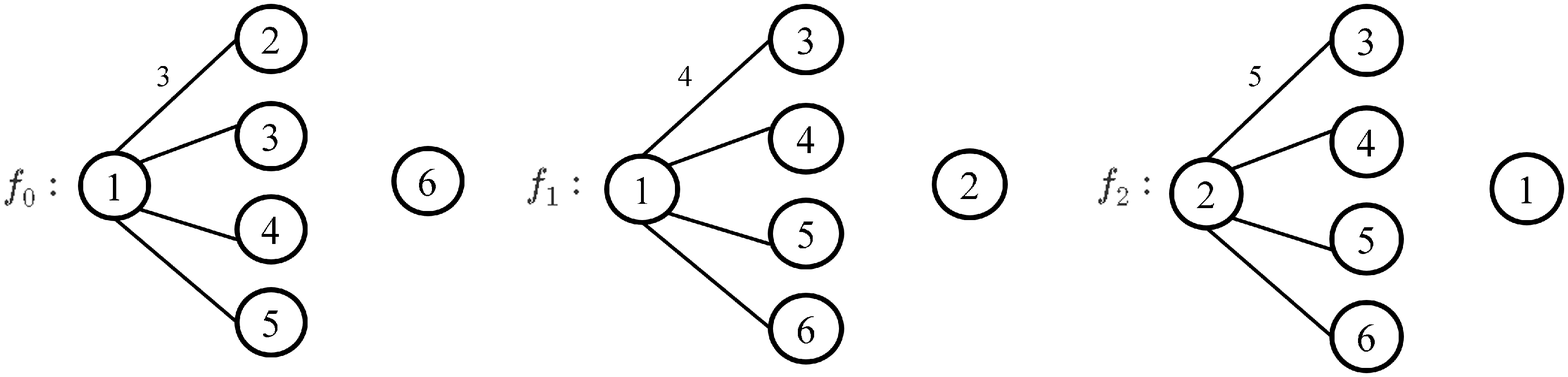}
    \caption{Example for labelings $f_{0},f_{1},f_{2}$ with minimum induced edge sums}
    \label{fig02}
  \end{center}
\end{figure}

Finally, recall that all valences from $2n+4$ to $3n+3l+3$ are attained by super edge-magic labelings. 
Also, all valences from $4n+2l+2$ to $4n+3l+2$ are attained by edge-magic labelings.
Therefore, if we take $l$ large enough so that $3n+3l+3\geq 4n+2l+2$.
We conclude that $G$ is perfect edge-magic and hence $\mu_{p}\left(K_{1,n}\right)<+ \infty$.
\end{proof}

\section{Conclusions and new research trends}

The main goal of this paper is to study the valences of the edge-magic and super edge-magic labelings of graphs. 
This study started with the paper by Godbold and Slater \cite{GS} in which they conjectured that the cycles other than $C_{5}$ are perfect edge-magic. 
This conjecture remains open, but in \cite{ McQuillan}, and independently in \cite{LMR5} some substantial progress was made. 
For further information on this problem, the interested reader may consult also \cite{LM} and \cite{Mohan}. 
In this paper, we have introduced the concepts of perfect edge-magic deficiency and perfect super edge-magic deficiency of graphs, and we have studied these concepts in relation to the star $K_{1,n}$. 
In addition, we have presented results on the cardinality of the super edge-magic set of graphs with certain order, size and girth. 

For future work, it is interesting to notice the following.
Consider the cycle $C_{3}$. 
It is clear that the super edge-magic interval for $C_{3}$ is $\left[9,9\right]$.
\begin{figure}[ht]
  \begin{center}
    \includegraphics[keepaspectratio, width=0.45\linewidth]{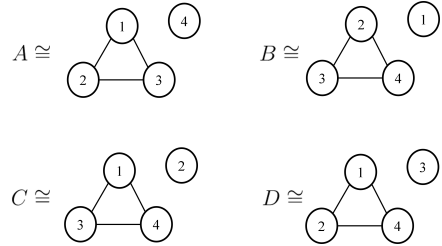}
    \caption{Four non-isomorphic bijections of the form $f:V\left(C_{3}\cup K_{1}\right) \rightarrow \left[1,4\right]$}
    \label{fig03}
  \end{center}
\end{figure}
Since $C_{3}$ is super edge-magic, it follows that $C_{3}$ is also perfect super edge-magic.
Thus, $\mu_{p}^{s}\left(G\right)=0$.

At this point, consider the graph $C_{3}\cup K_{1}$.
Since $C_{3}=K_{3}$, it follows that there exist four non-isomorphic bijections of the form 
$f:V\left(C_{3}\cup K_{1}\right) \rightarrow \left[1,4\right]$ (see Figure \ref{fig03}).
From these four bijections, it is clear that only bijections $A$ and $B$ can be extended to a super edge-magic labeling.
In case of $A$, the labeling has valence $10$ and in case of $B$ the labeling has valence $12$.
Thus, there is not any super edge-magic labeling of $C_{3}\cup K_{1}$ with valence $11$, implying that $C_{3}\cup K_{1}$ is not perfect super edge-magic.

From this, we can see that, on the contrary of when we deal with super edge-magic deficiency, that a graph $G$ has perfect super edge-magic deficiency $t$ so that $G\cup tK_{1}$ is perfect super edge-magic and $ C_{3}\cup t ^{\prime}K_{1}$ is not perfect super edge-magic if $ t^{\prime}<t$, it is not necessarily true that $G\cup t^{\prime\prime}K_{1}$ is also perfect super edge-magic for $t^{\prime\prime}>t$.
This suggests the definition of strong perfect super edge-magic deficiency.
The \emph{strong perfect super edge-magic deficiency} of a graph $G$ is the minimum nonnegative integer $t$ such that $G\cup t^{\prime\prime}K_{1}$ is perfect super edge-magic for all $t^{\prime\prime} \geq t$.
If such $t$ does not exist, then the strong super edge-magic deficiency of $G$ is defined to be $ + \infty$. 

We suspect that something similar comes about in the case of perfect edge-magic deficiency and that similar concepts can be introduced in this case; however, we do not have examples at this point to support this claim.
Other problems that we feel that would be interesting are problems of the following type.

\begin{problem}
For which real numbers $x$ $(0<x<1)$, there exists a sequence of graphs $(G_{1}^{x},G_{2}^{x}, \dots)$ such that $\text{lim}_{n \rightarrow +\infty} \frac{|\tau_{G_{n}^{x}}|}{|\lambda_{G_{n}^{x}}|} = x$?
\end{problem}
\begin{problem}
For which real numbers $x$ $(0<x<1)$, there exists a sequence of graphs $(H_{1}^{x},H_{2}^{x}, \dots)$ such that $\text{lim}_{n \rightarrow +\infty} \frac{|\sigma_{H_{n}^{x}}|}{|I_{H_{n}^{x}}|} = x$?
\end{problem}

In summary, we consider that the continuation of the ideas established in this paper constitute very interesting new trends for developing further research.

\section*{Acknowledgement}
The authors acknowledge Susana Clara L\'{o}pez Masip for her careful reading of this work and for her continuous support during the competition of this paper.

\end{document}